\documentclass[a4paper, oneside, 11pt]{amsart} 

\usepackage[english]{babel}
\usepackage[utf8]{inputenc}
\usepackage[T1]{fontenc}

\usepackage{amsmath}
\usepackage{amssymb}
\usepackage{amscd}
\usepackage{enumerate}

\usepackage{geometry}
\usepackage{verbatim}
\usepackage{graphicx}
\usepackage{color}
\usepackage{hyperref}

\newtheorem{thm}{Theorem}[section]

\newtheorem{lem}[thm]{Lemma}
\newtheorem{prop}[thm]{Proposition}
\newtheorem{dfn}[thm]{Definition}
\newtheorem*{q}{Question}
\newtheorem{rem}[thm]{Remark}

\newcommand\R{\mathbb R} 
\newcommand\Z{\mathbb Z} 

\renewcommand{\d}{\mathrm d}
\newcommand{\del}[1]{\partial #1}

\renewcommand{\leq}{\leqslant}
\renewcommand{\geq}{\geqslant}

\newcommand{\D}{\operatorname{D}}
\newcommand{\T}{\operatorname{T}}
\renewcommand{\L}{\operatorname{L}}
\renewcommand{\S}{\operatorname{S}}
\renewcommand{\H}{\operatorname{H}}
\newcommand{\U}{\operatorname{U}}
\DeclareMathOperator{\interior}{int}
\DeclareMathOperator{\id}{id}
\DeclareMathOperator{\Wh}{Wh}

\newcommand{\can}{\mathrm{can}}
\renewcommand{\O}{\operatorname{O}}

\title{Legendrian submanifolds with Hamiltonian isotopic symplectizations}

\author{Sylvain Courte}
\address{Uppsala Universitet, Sweden}
\email{sylvain.courte@math.uu.se}
\urladdr{http://www2.math.uu.se/~sylco859/}

\begin{document}

\begin{abstract}
In any contact manifold of dimension $2n-1\geq 11$, we construct examples of closed Legendrian submanifolds which are not diffeomorphic but whose Lagrangian cylinders in the symplectization are Hamiltonian isotopic.
\end{abstract}

\maketitle

\tableofcontents

\section{Introduction}

Let $(M, \xi)$ be a contact manifold ($\xi$ is cooriented) and denote by $\S M$ its symplectization, i.e. the set of covectors in $\T^* M$ whose kernel is equal (as cooriented hyperplane) to $\xi$, it comes with a natural projection $\pi : \S M \to M$ which is an $\R$-principal bundle (the $\R$-action is given by multiplying covectors by $e^t$ for $t\in \R$). To any Legendrian submanifold $\Lambda \subset M$, there corresponds its \emph{symplectization} $\S \Lambda =\pi^{-1}(\Lambda)$ which is a Lagrangian submanifold diffeomorphic to $\R \times \Lambda$. Any $\R$-equivariant Hamiltonian isotopy of $\S M$ that takes $\S \Lambda$ to $\S \Lambda'$ induces a contact isotopy of $M$ that takes $\Lambda$ to $\Lambda'$. However, if we forget about $\R$-equivariance, we are lead to consider the following question.

\begin{q}
If $\S \Lambda$ and $\S \Lambda'$ are Hamiltonian isotopic, does it follow that $\Lambda$ and $\Lambda'$ are Legendrian isotopic ?
\end{q}

This is a relative version of the question whether contact manifolds with exact symplectomorphic symplectizations are necessarily contactomorphic. The latter question was answered negatively in \cite{courte_contact_2014} and we explain in this paper that the same phenomenon arises in this case.

\begin{thm}\label{thm:main}
In any closed contact manifold $(M, \xi)$ of dimension $2n-1\geq 11$, there exist closed Legendrian submanifolds which are not diffeomorphic but whose symplectizations are Hamiltonian isotopic.
\end{thm}

This theorem will follow from a general construction using Lagrangian h-cobordisms and a Mazur trick argument. An essential ingredient in the proof is the notion of flexible Lagrangian cobordisms recently introduced by Eliashberg, Ganatra and Lazarev in \cite{eliashberg_flexible_2015}.

\section{Exact Lagrangian cobordisms and the Mazur trick}

Let $(M, \xi)$ be a closed connected contact manifold, recall that its symplectization $\S M$ is equipped with canonical Liouville vector vield $X_{\can}$ and Liouville form $\lambda_{\can}$ (the restrictions of those of $\T^* M$) and that a contact form for $(M, \xi)$ is a section of the bundle $\S M \to M$. We denote by $\S M^{\geq \alpha}$ the subset of $\S M$ above the section $\alpha$ and use obvious notations for similar subsets of $\S M$ or subsets of a Lagrangian cylinder $\S \Lambda$.

\begin{dfn}\label{def:lagcob}
An \emph{exact Lagrangian cobordism} in $\S M$ is a Lagrangian submanifold $L \subset \S M$ such that there exists two sections $\alpha_-$ and $\alpha_+$ with $\alpha_-<\alpha_+$ at each point of $M$ with the following properties :
\begin{enumerate}
\item There exists two closed Legendrian submanifolds $\Lambda_-$ and $\Lambda_+$ such that 
\[L \cap \S M^{\geq \alpha_+}=\S \Lambda_+^{\geq \alpha_+} \text{ and } L \cap \S M^{\leq \alpha_-}=\S \Lambda_-^{\leq \alpha_-}.\]
\item The region $L \cap \S M^{[\alpha_-, \alpha_+]}$ is a compact cobordism from $\Lambda_-$ to $\Lambda_+$ (without any other boundary).
\item Denoting $i : L \to \S M$ the inclusion, there exists a function $g : L \to \R$ with $i^*\lambda_{\can}=\d g$ which is constant on $L \cap \S M^{\geq \alpha_+}$ and on $L \cap \S M^{\leq \alpha_-}$.
\end{enumerate}
\end{dfn}

\begin{rem}
The function $g$ in definition \ref{def:lagcob} can be extended to $\S M$ as a function (still denoted by $g$) constant on $\S M^{\geq \alpha_+}$ and on $\S M^{\leq \alpha_-}$. The Liouville vector field $X=X_{\can}+X_g$ \footnote{The Hamiltonian vector field $X_g$ is defined by $X_g \lrcorner \omega = - \d g$.} is then tangent to $L$ and coincides with $X_{\can}$ on $\S M^{\geq \alpha_+} \cup \S M^{\leq \alpha_-}$. We say that such a vector field is \emph{adapted} to $L$.
\end{rem}

\begin{rem}\label{rem:stability}
If $\phi$ is a diffeomorphism of $\S M$ that preserve $\lambda_{\can}$ at infinity, then it lifts contact diffeomorphisms $\phi_-$ and $\phi_+$ near $- \infty$ and $+ \infty$ \footnote{By that we mean, above or below some section of $\S M$.} respectively and it is automatically exact ($\phi^*\lambda_{\can}-\lambda_{\can}$ is exact). These diffeomorphisms form a group denoted by $\mathcal{G}$, the subgroup defined by $\{\phi_-=\id, \phi_+=\id\}$ will be denoted by $\mathcal{G}_{\partial}$. The image of an exact Lagrangian cobordism $(L; \Lambda_-, \Lambda_+)$ by $\phi \in \mathcal{G}$ is then an exact Lagrangian cobordism $(\phi(L); \phi_-(\Lambda_-), \phi_+(\lambda_+))$. Exact Lagrangian cobordisms are stable in the following sense : any one-parameter family $L_t$, $t\in[0,1]$, can be written $\phi_t(L_0)$ where $\phi_t \in \mathcal{G}$, $\phi_0=\id$; moreover if $L_t$ is constant at $-\infty$ and at $+\infty$, we can require $\phi_t$ to lie in $\mathcal{G}_{\partial}$.
\end{rem}

\begin{dfn}
Two exact Lagrangian cobordisms $(L_0; \Lambda, \Lambda')$ and $(L_1; \Lambda, \Lambda')$ in $\S M$ are said to be \emph{equivalent} (what we write $\L_0 \sim L_1$) if there exists a Hamiltonian isotopy $\phi_t : \S M \to \S M$, $t\in[0,1]$, and two sections $\alpha_-<\alpha_+$ of $\S M$ such that $\phi_0=\id$, $\phi_1(L_0)=L_1$ and $\phi_t$ equals the identity on $\S M^{\geq\alpha_+} \cup \S M^{\leq \alpha_-}$ (that is $\phi_t \in \mathcal{G}_\del$ with the notations above; according to remark \ref{rem:stability}, this is the same as being isotopic relative boundary through exact Lagrangian cobordisms).
\end{dfn}

Exact Lagrangian cobordisms can be composed : given such $(L; \Lambda, \Lambda')$ and $(L'; \Lambda', \Lambda'')$ we have sections $\alpha$ and $\alpha'$ such that $L \cap \S M^{\geq \alpha} = \S \Lambda'^{\geq \alpha}$ and $L' \cap \S M ^{\leq \alpha'} = \S \Lambda'^{\leq \alpha'} $. If we can find such sections with $\alpha < \alpha'$, then $L$ and $L'$ can naturally be glued because they both coincide with $\S \Lambda'$ in $\S M^{[\alpha, \alpha']}$, now observe that we can always achieve this condition by pushing up $L'$ along the flow $\varphi_t$ of $X_{\can}$. We denote by $L \odot L'$ the resulting exact Lagrangian cobordism. This composition operation satisfies the following properties.

\begin{enumerate}
\item The equivalence class of $L\odot L'$ is independent of choices and depends only on the equivalence classes of $L$ and $L'$.
\item $L \odot \S \Lambda' \sim L$ and $\S \Lambda \odot L \sim L$. 
\item The composition is associative on equivalence classes, that is $L \odot (L' \odot L'') \sim (L\odot L') \odot L''$.
\item Given a sequence $(L_i; \Lambda_i, \Lambda_{i+1})$ for $i \in \Z$ of exact Lagrangian cobordisms, we can construct the infinite composition $\bigodot_{i \in \Z} L_i$ whose Hamiltonian isotopy class (not with compact support) is independent of choices and only depends on the equivalence class of each $L_i$.
\end{enumerate}

\begin{dfn}
An exact Lagrangian cobordism $(L; \Lambda, \Lambda')$ is said to be \emph{invertible} if there exists another exact Lagrangian cobordism $(L'; \Lambda', \Lambda)$ such that $L \odot L' \sim \S \Lambda$ and $L' \odot L \sim \S \Lambda'$.
\end{dfn}

\begin{rem}\label{rem:inverse}
By associativity of composition, if $L \odot L' \sim \S \Lambda$ and $L' \odot L'' \sim \S \Lambda'$, then $L\sim L''$ and $L$ is invertible.
\end{rem}

\begin{prop}\label{prop:mazur}
Let $\Lambda$ and $\Lambda'$ be closed Legendrian submanifolds of a closed contact manifold $(M, \xi)$. The following assertions are equivalent:
\begin{enumerate}
\item $\S \Lambda$ and $\S \Lambda'$ are Hamiltonian isotopic.
\item There exists an \emph{invertible} exact Lagrangian cobordism $(L; \Lambda, \Lambda')$.
\end{enumerate}
\end{prop}

\begin{proof}
$(1) \Rightarrow (2)$: Let $H_t : \S M \to \R$ be a Hamiltonian generating an isotopy $\phi_t$, $t\in [0,1]$, of $\S M$ such that $\phi_0=\id$, $\phi_1(\S\Lambda)=\S \Lambda'$. We pick four sections $\alpha_1<\alpha_2<\alpha_3<\alpha_4$ and two functions $\rho, \rho' : \S M \to [0,1]$ with the following properties:
\begin{itemize}
\item $\rho=1$ in $\S M^{\geq \alpha_2}$ and $\rho=0$ in $\S M^{\leq \alpha_1}$,
\item $\rho'=1$ in $\S M^{\leq \alpha_3}$ and $\rho'=0$ in $\S M^{\geq \alpha_4}$.
\end{itemize}

Denote respectively by $\psi_t$, $\psi'_t$ and $\theta_t$ the Hamiltonian isotopies generated respectively by $\rho H_t$, $\rho' H_t$ and $\rho\rho'H_t$ (these are all well defined for $t \in [0,1]$). Then $L=\psi_1(\S \Lambda)$ and $L'=\psi'_1(\S \Lambda)$ are exact Lagrangian cobordisms respectively from $\Lambda$ to $\Lambda'$ and from $\Lambda'$ to $\Lambda$. Moreover, if we chose $\alpha_3/\alpha_2$ sufficiently big, then $L \odot L'$ sits naturally in $\S M$ as $\theta_1(\S \Lambda)$ and is equivalent to $\S \Lambda$ (via the isotopy $\theta_t$). We can likewise construct a right inverse for $L'$ and we conclude using remark \ref{rem:inverse}.

\hfill\\
$(2) \Rightarrow (1)$ : Let $(L';\Lambda', \Lambda)$ be an inverse for $(L; \Lambda, \Lambda')$ and consider the infinite composition
\[L_\infty = \dots \odot L\odot L' \odot L \odot L' \odot \dots \]

By introducing parentheses in two different ways ($(L\odot L')$ or $(L' \odot L)$) in the above expression, we get that $L_\infty$ is Hamiltonian isotopic to $\S \Lambda$ as well as to $\S \Lambda'$.
\end{proof}

\begin{rem}
It follows from proposition \ref{prop:mazur} together with functoriality properties of symplectic field theory that such Legendrian submanifolds have isomorphic Legendrian contact homology.
\end{rem}

Our goal is now to construct non-trivial invertible Lagrangian cobordisms.

\section{Flexible Lagrangian h-cobordisms}

Let $(M,\xi)$ be a contact manifold of dimension $2n-1 \geq 5$.

\begin{dfn}[\cite{eliashberg_flexible_2015}]
An \emph{exact Lagrangian cobordism} $L \subset \S M$ is called \emph{regular} if there exists an adapted Liouville vector field $X$ and a proper Morse function $f : \S M \to \R$ for which $X$ is a pseudo-gradient. Moreover if there exists such an \emph{adapted pair} $(f,X)$ for which $f$ is excellent (all critical values are distinct) and the attaching spheres of critical points of index $n$ are loose (see \cite{murphy_loose_2012}) in the complement of $L$, then $L$ (as well as the pair $(f,X)$) is said to be \emph{flexible}.
\end{dfn}

Note that the critical points of $f|L$ are necessarily critical points of $f$ and, in the flexible case, there cannot be any critical point of index $n$ on $L$. The definition can obviously be extended to Lagrangian cobordisms into arbitrary flexible Weinstein cobordisms.

Recall that an \emph{h-cobordism} is a cobordism which deformation retracts on its bottom boundary as well as on its top boundary. According to the s-cobordism theorem (see \cite{kervaire_theoreme_1965}), h-cobordisms from a given closed manifold $M$ are classified up to diffeomorphism relative to $M$ by so-called \emph{Whitehead torsion}, an invariant which takes values in the Whitehead group $\Wh(M)$ of $M$ (it actually depends only on $\pi_1 M$). Essentially since each element in a group has an inverse, $h$-cobordisms of dimension $\geq 6$ are invertible for the composition of cobordisms (see \cite{stallings_infinite_1965}).

\begin{thm}\label{thm:inverse}
Let $(M,\xi)$ be a closed contact manifold of dimension $\geq 11$.
\begin{enumerate}
\item Let $\Lambda$ a closed Legendrian submanifold in $M$, and $(L; \Lambda, \Lambda')$ an h-cobordism. Then $L$ can be embedded in $\S M$ has a flexible Lagrangian cobordism starting from $\Lambda$.
\item Any flexible Lagrangian h-cobordism in $\S M$ is invertible (as an exact Lagrangian cobordism).
\end{enumerate}
\end{thm}

We need a couple of lemmas. The frst one is proved in \cite{eliashberg_flexible_2015}, proposition 2.5.

\begin{lem}\label{lem:index}
For any regular Lagrangian cobordism $L$ together with an adapted pair $(f,X)$, we can find a homotopy $(f_t, X_t)$ of adapted pairs such that $(f_0, X_0)=(f,X)$ and for all critical point of $f_1$ on $L$ the index is the same for $f_1$ and $f_1|L$. Moreover if $(f_0,X_0)$ is flexible, we can require $(f_t,X_t)$ to be flexible for all $t$.
\end{lem}

\begin{lem}\label{lem:elimination}
Let $(M, \xi)$ be a contact manifold of dimension $\geq 5$. Let $(L; \Lambda, \Lambda')$ be a flexible Lagrangian cobordism of $\S M$ which is diffeomorphic to $\Lambda \times [0,1]$, then there exists an adapted pair without critical points.
\end{lem}

\begin{proof}
We start with a flexible adapted pair $(f,X)$. By lemma \ref{lem:index}, we can assume that the critical points on $L$ have same index for $f|L$ and $f$. Since there are no $X$-trajectories going from critical points outside of $L$ to critical points on $L$, we can reorder the critical values so that the critical points on $L$ lie below all the others. Since $L$ is diffeomorphic to $\Lambda \times [0,1]$, the function $g=f|L$ can be deformed via a homotopy $g_t$, $t\in [0,1]$, to a function without critical points and moreover this can be done without introducing any maximum along the deformation. We then extend the homotopy $g_t$ to a homotopy $(f_t, X_t)$ of flexible adapted pairs supported into an arbitrary small neighbourhood of the support of the homotopy $g_t$ (see \cite{cieliebak_stein_2012} lemma 12.8). We then proceed to the cancellation of the remaining critical points which are all outside of $L$, following the proof of the h-cobordism theorem :
\begin{itemize}
\item Cancel index $0$ critical points with some index $1$ critical points.
\item Trade critical points of index $i$ for critical points of index $i+2$, until there only remains critical points of index $n-1$ and $n$.
\item Cancel together critical points of index $n-1$ and $n$.
\end{itemize}

We have to go through these steps keeping $(f,X)$ fixed near $L$. We claim this is possible because every $X$-trajectory between critical points are disjoint from $L$. The main point to notice is that the isotopies of the attaching spheres needed to arrange cancellation positions can be done in the complement of $L$ because they can be localized near Whitney $2$-disks which are generically disjoint from $L$. Subcriticallity or looseness in the complement of $L$ then allows to realize this isotopies as isotropic isotopies as in \cite{cieliebak_stein_2012} chapter 14 (see lemma 14.10 for example).
\end{proof}

\begin{proof}[Proof of theorem \ref{thm:inverse}]
(1) Recall that any h-cobordism of dimension at least 6 can be presented with a Morse function having ony critical points of index $2$ and $3$ (see \cite{kervaire_theoreme_1965}). We first construct a flexible Weinstein cobordism $(W;M, M')$ containing a flexible Weinstein Lagrangian cobordism $(L; \Lambda, \Lambda')$ by attaching Weinstein handles of index $2$ and $3$ on $\Lambda$. Denoting by $\tau \in \Wh(L)$ the Whitehead torsion of $L$, we note that the ambiant cobordism $W$ is also an h-cobordism and its torsion is $i(\tau)$ where $i : \Wh(\Lambda) \to \Wh(M)$ is the map induced by inclusion. We now attach handles of index $2$ and $3$ on top of $M'$ away from $\Lambda'$ to produce a flexible Weinstein h-cobordism $W'$ with torsion $-i(\tau) \in \Wh(M')$ (we identify $\Wh(M) \simeq \Wh(M')$ via the homotopy equivalence induced by $W$). The Lagrangian $L$ can be continued inside of $W'$ by composing with the Lagrangian cylinder $\S \Lambda'$. The composition $W \odot W'$ is a flexible Weinstein cobordism and it is diffeomorphic to $M \times [0,1]$ since its Whitehead torsion vanishes. We can therefore cancel all the handles and show that $W \odot W'$ is equivalent to $\S M$ relative to the negative boundary (see \cite{cieliebak_stein_2012} corollary 14.2). Thus $L$ now sits as a flexible Lagrangian cobordism in $\S M$.
\hfill\\

(2) Let $(L'_1; \Lambda', \Lambda)$ be an inverse cobordism for $(L; \Lambda, \Lambda')$. Using the first point, we can embed $L'$ as a flexible Lagrangian cobordism in $\S M$. Denote by $\Lambda_1$ the positive Legendrian boundary of $L'_1$, note that it is a priori different from $\Lambda$. Now lemma \ref{lem:elimination} allows to find an adapted pair $(f, X)$ without critical points for the composition $L\odot L'_1$. By sending the trajectories of $X_{\can}$ to that of $X$ we find a symplectic pseudo-isotopy $\psi$ of $\S M$ (that is $\psi \in \mathcal{G}$ with $\psi_-=\id$) that takes $\S \Lambda$ to $L\odot L_1'$. We undo this pseudo-isotopy by composing $L'_1$ further with $L'_2=\psi^{-1}(\S \Lambda_1)$, we then get a flexible Lagrangian h-cobordism $L'=L'_1 \odot L'_2$ from $\Lambda'$ to $\Lambda$ such that $L \odot L'$ is equivalent to $\S \Lambda$. We can repeat the same argument to produce a right inverse for $L'$ and the result now follows from remark \ref{rem:inverse}.
\end{proof}

\begin{rem}
Starting from an exact Lagrangian filling $F$ of a Legendrian $\Lambda$, the same method shows that $F$ is Hamiltonian isotopic to the composition of $F$ with any flexible Lagrangian h-cobordism starting from $\Lambda$.
\end{rem}

\section{Examples}
\paragraph{\textbf{An example where $\Lambda$ and $\Lambda'$ are not diffeomorphic}}

For $n\geq 6$, consider the manifold $\Lambda = \L(4,1) \times \T^{n-4}$. It was proved in \cite{farrell_h-cobordant_1967}, that there exists an h-cobordism $(L; \Lambda, \Lambda')$ such that $\Lambda'$ is not diffeomorphic to $\Lambda$. We claim that $\Lambda$  admits a Legendrian embedding into $\R^{2n-1}$ endowed with its standard contact structure. Indeed, $\Lambda$ is parallelizable so we can find a Legendrian bundle monomorphism $\T \Lambda \to \R^{2n-1}$ and then turn it into a Legendrian embedding via Gromov's h-principle (see \cite{eliashberg_introduction_2002} theorem 16.1.3, and note that a generic Legendrian immersion is an embedding). This Legendrian embedding of $\Lambda$ can be implanted in any contact manifold via a Darboux chart. Theorem \ref{thm:main} now follows from theorem \ref{thm:inverse} and proposition \ref{prop:mazur}.
\hfill\\
\paragraph{\textbf{An example where $\Lambda$ and $\Lambda'$ are smoothly isotopic but not Legendrian isotopic}}
The following construction is very similar to that in \cite{courte_contact_2015} section 3, but we repeat some of the arguments there for the convenience of the reader.

Consider the closed $7$-dimensional manifold $\Lambda=\L(5,1) \times \S^4$. Note that $\Lambda$ is parallelizable and that $\pi_3 \Lambda = \pi_3 \L(5,1)=\Z$ (a generator is given by the universal covering map $\S^3 \to \L(5,1)$).

\begin{lem}\label{lem:pi3}
\begin{enumerate}
\item There exists an h-cobordism $(L; \Lambda, \Lambda)$ such that the induced map $f : \Lambda \to \Lambda$ acts by multiplication by $-1$ on $\pi_3 \Lambda$.
\item No diffeomorphism of $\Lambda$ may act by multiplication by $-1$ on $\pi_3 \Lambda$.
\end{enumerate}
\end{lem}
\begin{proof}
(1): There are exactly two homotopy classes of maps $\L(5,1) \to \L(5,1)$ of degree $-1$ (these are automatically homotopy equivalences) and they respectively induce multiplication by $2$ and $-2$ on $\pi_1\L(5,1)=\Z/5\Z$ (see \cite{cohen_course_1973}, 29.5). We pick such a map and perturb it to an embedding $j:\L(5,1) \to \L(5,1) \times \interior \D^5$ using Whitney's embedding theorem. The normal bundle of $j$ is trivial because it is stably trivial and has rank greater than the dimension of the base. We can therefore extend $j$ to an embedding $\L(5,1) \times \D^5 \to \L(5,1) \times \interior \D^5$ that we still denote by $j$. The region $L=\L(5,1)\times \D^5 \setminus j(\L(5,1) \times \interior \D^5)$ is an h-cobordism from $\Lambda$ to itself (see \cite{milnor_two_1961} lemma 2 p.579). The map $f : \Lambda \to \Lambda$ induced by the cobordism $L$ can be defined as $f=r\circ i$ where $i : \Lambda \to L$ is the inclusion of the negative boundary and $r : L \to \Lambda$ is a deformation retraction on the positive boundary (the homotopy class of $f$ is independent of choices). Since we started with a map of degree $-1$ on $\L(5,1)$, we see that $j$ induces multiplication by $-1$ on $\H_3(\L(5,1)\times \D^5; \Z)\simeq \Z$ as well as on $\pi_3(\L(5,1) \times \D^5) \simeq \Z$ because the Hurewicz homomorphism $\pi_3\L(5,1) \to \H_3(\L(5,1); \Z)$ is non zero. It follows from the commutativity up to homotopy of the following diagram (the vertical arrows are obvious inclusions)
\[\begin{CD}
\L(5,1) \times \D^5     @>j>> \L(5,1) \times \D^5  \\
@AAA        @AAA\\
\Lambda     @>f>>  \Lambda
\end{CD}\]
that the map $f$ also induces multiplication by $-1$ on $\pi_3 \Lambda$.
\hfill\\
(2): If $\psi : \Lambda \to \Lambda$ was such a diffeomorphism, then the map $\L(5,1) \to \L(5,1)$, obtained by composing the inclusion of a factor with $\psi$ and then projection, would have degree $-1$. But then $\psi$ necessarily acts by multiplication by $\pm 2$ on $\pi_1$, in which case the Whitehead torsion of $\psi$ must be non zero (see \cite{courte_contact_2015} lemma 3.2) contradicting the fact that $\psi$ is a diffeomorphism.
\end{proof}

Let $(L; \Lambda, \Lambda)$ be an h-cobordism given by the lemma above. We fix a framing of $\Lambda$ and extend it to a framing of $L$ by using an isomorphism $\T L \to  \R \times \T \Lambda$ lifting the retraction map $r : L \to \Lambda$ on the positive boundary. Note that the induced framing of $\T \Lambda \times \R$ on the negative boundary a priori differs from the given one : it is the image of the given framing by a map $A : \Lambda \to \O(8)\subseteq \U(8)$. Recall that any Legendrian immersion $\Lambda \to \R^{15}$ gives rise to a map $\Lambda \to \U(7)$ well-defined up to homotopy and Gromov's h-principle (see \cite{eliashberg_introduction_2002} theorem 16.1.3) implies that this classifies Legendrian regular homotopy classes. Given an embedding of $L$ as a Lagrangian cobordism in $\S \R^{15}$, we get maps $g:\Lambda \to \U(7)$, $g':\Lambda \to \U(7)$ and $G : L \to \U(8)$ associated respectively to $\del_-L$, $\del_+ L$ and $L$. These maps are related by the following formulas:

\[A.s\circ g \sim G \circ i, \quad s\circ g' \circ r \sim G\]
where $\sim$ here means homotopic, $s:\U(7) \to \U(8)$ is the stabilization map (note that this is an isomorphism on $\pi_3$), $r,i$ are defined as in the proof of the lemma \ref{lem:pi3} and the dot denotes multiplication in $\U(8)$. In particular, we get $s\circ g'$ out of $s\circ g$ :
\[s\circ g' \sim A. s\circ g\circ f^{-1}.\]
Recall from Bott periodicity that $\pi_3 \U(8)\simeq\Z$. Identifying $\pi_3 \Lambda$ and $\pi_3 \U(8)$ with $\Z$, the map induced on $\pi_3$ by $s\circ g$, $s\circ g'$ and $A$ are respectively multiplication by integers $b,b'$ and $a$ and the equation above reads:
\[b'=a-b\]
(note that multiplication on $\U(8)$ induces addition on $\pi_3 \U(8)$).

We now observe that, whatever $a$ is, we can choose $g$ such that $b' \neq b$ and therefore $g'$ is not homotopic to $g$. Indeed 
\begin{itemize}
\item if $a\neq 0$, we take $g$ to be constant so that $b=0$ and $b' \neq 0$,
\item if $a=0$, we take $g=\alpha \circ h \circ p_1$ where $p_1 : \Lambda \to \L(5,1)$ is the projection on the first factor, $h : \L(5,1) \to \S^3$ is a map of degree $1$ and $\alpha: \S^3 \to \U(7)$ corresponds to $1 \in \Z=\pi_3 \U(7)=\pi_3\U(8)$, so that $b=5$ and $b'=-5$.
\end{itemize}

The rest of the construction is the same as in the first example above: we take a Legendrian embedding $\phi : \Lambda \to \R^{15}$ that induces the map $g$ and use theorem \ref{thm:inverse} to obtain an embedding of $L$ as a flexible Lagrangian cobordism in $\S\R^{15}$ with negative boundary $\phi$ and a new Legendrian embedding $\phi' : \Lambda \to \R^{15}$ on the positive boundary which induces the map $g'$. The Legendrian embeddings $\phi$ and $\phi'$ are not homotopic through Legendrian immersions and moreover using the second point of lemma \ref{lem:pi3}, we see that this cannot be arranged by composing $\phi'$ by a diffeomorphism of $\Lambda$. Hence the Legendrian submanifolds $\phi(\Lambda)$ and $\phi'(\Lambda)$ are not Legendrian isotopic though they have Hamiltonian isotopic symplectizations and by Haefliger's embedding theorem (see \cite{haefliger_plongements_1961}) they are smoothly isotopic.

\hfill\\

\paragraph{\textbf{Acknowledgments}}The author would like to thank Yasha Eliashberg and Tobias Ekholm for encouraging discussions, and also Thomas Kragh and Rémi Crétois for useful discussions concerning the second example in section 4. He acknowledges support from the Knut and Alice Wallenberg Foundation.

\bibliographystyle{amsalpha}
\bibliography{biblio}

\end{document}